\documentclass[preprint,12pt]{elsarticle}

\usepackage{amsfonts,amsthm,amsmath}
\usepackage{mathrsfs}
\usepackage{vaucanson-g}
\usepackage{gastex}

\def\A{\mathbb{A}}

\newcommand{\reals}{{\mathbb R}}

\newcommand{\ints}{{\mathbb Z}}

\DeclareMathOperator{\Ff}{Fact}
\DeclareMathOperator{\Pre}{Pref}

\newcommand{\R}{\mathcal{R}}
\newcommand{\D}{\mathcal{D}}
\newcommand{\cP}{\mathcal{P}}
\newcommand{\cPret}{\mathcal{P}_{\rm{ret}}}
\newcommand{\FP}{\mathcal{FP}^{\prime}}

\newtheorem{thm}{Theorem}
\newtheorem{conjecture}{Conjecture}
\newtheorem{rem}[thm]{Remark}

\newtheorem{theorem}{Theorem}[section]
\newtheorem{lemma}[theorem]{Lemma}
\newtheorem{corollary}[theorem]{Corollary}
\newtheorem{proposition}[theorem]{Proposition}

\theoremstyle{definition}

\begin{document}

\begin{frontmatter}

\title{On a question of Hof, Knill and Simon on palindromic substitutive systems}

\author[label1]{Tero Harju}
  \ead{harju@utu.fi}

  \author[label1]{Jetro Vesti}
  \ead{jejove@utu.fi}

   \author[label1,label2]{Luca Q. Zamboni\fnref{label3}}
  \ead{lupastis@gmail.com}

  \fntext[label3]{Partially supported by a FiDiPro grant (137991)
  from the Academy of Finland and by
ANR grant {\sl SUBTILE}.}

\address[label1]{FUNDIM, University of Turku, Finland}

\address[label2]{Universit\'e de Lyon,
Universit\'e Lyon 1, CNRS UMR 5208, Institut Camille Jordan, 43 boulevard du 11 novembre
1918, F69622 Villeurbanne Cedex, France}

\begin{abstract}  In a 1995 paper,  Hof, Knill and Simon obtain a sufficient combinatorial
criterion on the hull $\Omega$ of the potential of a discrete Schr\"odinger operator which
guarantees purely singular continuous spectrum on a generic subset of $\Omega.$
In part, this condition requires the existence of infinitely many palindromic factors.
In this same paper, they introduce the class P of morphisms $f:A^*\rightarrow B^*$ of
the form $a\mapsto pq_a$ and ask whether every palindromic subshift generated by a
primitive substitution arises from morphisms of class P or by morphisms of the form
$a\mapsto q_ap$ where again $p$ and $q_a$ are palindromes.
 In this paper we give a partial affirmative answer to the question of
Hof, Knill and Simon: we show that every rich primitive substitutive subshift  is generated
by at most two morphisms  each of which is conjugate to a morphism of class P. More
precisely, we show that every rich (or almost rich in the sense of finite defect) primitive
morphic word $y\in B^\omega$ is of the form $y=f(x)$ where $f:A^*\rightarrow B^*$ is
conjugate to a morphism of class P, and where $x$ is a rich word fixed by a primitive
substitution $g:A^*\rightarrow A^*$ conjugate to one in class P.
\end{abstract}


\end{frontmatter}

\section{Introduction}
Let $A$ be a finite non-empty set. Associated to each uniformly recurrent word $w \in
A^\omega$, is the subshift $\Omega=\Omega(w)$ of all two-sided infinite words having the
same factors as $w.$ In many interesting cases (including when $w$  is a Sturmian word,
or generated by a primitive substitution),  $\Omega(w)$ is a $1$-dimensional {\it
quasicrystal} modeled by a family of aperiodic words which are locally indistinguishable.
To each $x\in \Omega,$ one associates a discrete one-dimensional Schr\"odinger operator
$H_x$ which acts in the Hilbert space ${\mathcal{H}} = \ell^2(\ints)$. If $\phi \in
{\mathcal{H}}$, then $H_x \phi$ is given by
$$
(H_x\phi)(n) = \phi(n+1) + \phi(n-1) + f(x_n) \phi(n),
$$
where the potential $f: A\rightarrow \reals$ is any injective mapping. In case $f$ is
bounded,  $H_x$ becomes  a bounded self-adjoint operator. Given an initial state $\phi \in
\mathcal{H}$, the Schr\"odinger time evolution is given by $\phi(t) = \exp (-itH_x) \phi$,
where $\exp (-itH_x)$ is given by the spectral theorem. An important question in
connection with  the conductivity of the given structure is  whether $\phi(t)$  spreads out
in space, and if so, how fast. In this context, it is natural to consider the spectral measure
$\mu_\phi$ associated with $\phi$ defined by
$$
\langle \phi , (H_x - z)^{-1} \phi \rangle = \int_\R \frac{d\mu_\phi(x)}{x - z}
\mbox{ for every $z$ with } {\rm Im} \, z > 0.
$$
Roughly speaking, the more continuous $\mu_\phi$, the faster the spreading of $\phi(t)$;
compare, for example, \cite{bgt1,g3,l2}.

In physical terms, the spectral properties of $H_x$ determine the "conductivity properties"
of the given structure. Roughly, if the spectrum is absolutely continuous, then the structure
behaves like a conductor, while in the case of pure point spectrum, it behaves like an
insulator. An intermediate spectral type, known as {\it singular continuous spectrum}, is
expected to give rise to intermediate transport properties. For periodic structures, singular
continuous spectra does not occur. However, for one-dimensional quasicrystals, this
spectral type is experimentally seen to be rather common. In\cite{HKS}, Hof, Knill and
Simon give a sufficient combinatorial criterion for purely singular continuous spectrum in
terms of a strong palindromicity property of the underlying word. More precisely,  a word
$x\in A^\ints$ is said to be {\it strongly palindromic} if there exist $B>0$ and a sequence
$(u_i)_{i\geq 1}$ of palindromic factors of $x$ centered at $m_i\longrightarrow +\infty$
such that $e^{Bm_i}/|u_i| \longrightarrow 0.$  They then show that if $x\in A^\ints$ is
aperiodic and {\it palindromic}, meaning that $x$ contains infinitely many distinct
palindromes, then its subshift contains uncountably many strongly palindromic words (see
Proposition 2.1 in \cite{HKS}). It follows from a result of Jitomirskaya and Simon in
\cite{JS} that if $x$ is strongly palindromic, then the spectrum of $H_x$ is empty. They then
deduce that  if $\Omega$ is uniquely ergodic and generated by an aperiodic palindromic
word $w,$  then the operator $H_x$ has purely singular continuous spectrum for
uncountably many $x\in \Omega .$ In this same paper they introduce the class  $\cP$ of
(non-erasing) morphisms $f:A^*\rightarrow B^*$ of the form $a\mapsto pq_a$ where
$p,q_a$ are each palindromes. Morphisms in this class are said to be of class P. Actually in
\cite{HKS} they consider only primitive substitutions in $\cP.$ As was observed in
\cite{HKS}, substitutions in $\cP$ generate palindromic subshifts. They also point out that
substitutions of the form  $a\mapsto q_ap$ with $p$ and $q_a$ palindromes also generate
palindromic subshifts. Thus they extend $\cP$ to include also substitutions of the form
$f(a)=q_ap$ and remark:

\begin{rem}[Hof, Knill, Simon, Remark 3 in \cite{HKS}]\label{quest}
We do not know whether all palindromic subshifts generated by primitive substitutions
arise from substitutions in this extended class.
\end{rem}

Over the years this remark has evolved into what is now called the class P conjecture. The
first step in the evolution process, which is perhaps non consequential,  was to convert this
remark into a question. The second step, which in our minds represents a significant
alteration, was to replace the entire subshift by a single element within the subshift which
is fixed by a primitive substitution. The third  was to give a precise interpretation to ``arise
from" as meaning ``fixed by". The fourth and final step was to call it a conjecture.

If we agree to focus on a single element of the subshift, then it is natural to widen the class
of possible morphisms. In fact, if $x\in A^\omega$ is generated by a primitive substitution
$f:A\rightarrow A^+$ and each of the images $f(a)$ for $a\in A$ begins or ends in a
common letter, then one may conjugate each of the images by this common letter to obtain
a new primitive substitution which generates the same subshift as $f.$ Hence if $f$ is in
class P, then this new substitution need no longer be in class $\cP$ although it,  or some
power of it, will have a palindromic fixed point.  For instance,  Blondin Mass\'e proved that
the fixed point $x$ of the primitive substitution $a\mapsto abbab, b\mapsto abb$ is
palindromic, but that $x$ itself is not fixed by a primitive substitution in $\cP$  (see
Proposition 3.5 in \cite{Lab1}).  However  this morphism is conjugate to the class P
morphism $a\mapsto bbaba, b\mapsto bba.$ Thus it is reasonable to consider the class
$\cP'$ of all morphisms $f$ which are conjugate to some morphism in $\cP.$ Let $\FP$
denote the set of all infinite words which are fixed by some primitive substitution in class
$\cP'.$ Then the original remark of Hof, Knill and Simon was reformulated in terms of the
following conjecture, called the class P conjecture:

\begin{conjecture}[Blondin Mass\'e, Labb\'e in \cite{Lab1}]
If $x$ is a palindromic word fixed by a primitive substitution, then $x\in \FP.$
\end{conjecture}

Partial results in support of the  conjecture were obtained by Allouche et al. in case $x$ is
periodic (see \cite{ABCD}) and by Tan in case $x$ is a binary word (see \cite {Tan}). Tan
proves that if $x$ is a palindromic binary word fixed by a primitive substitution $f,$  then
$f^2\in \cP'.$

Recently Labb\'e \cite{Lab2} produced a counter-example to the class P conjecture on a
ternary alphabet. The counter-example is given by the fixed point
\[x=acabacacabacabacabacacabac\cdots\] of the primitive substitution: \[f: \,\,\,\,
a\mapsto ac,\,\,\,\, b\mapsto acab,\,\,\,\,  c\mapsto ab.\]  He proves that $x$ is
palindromic but not in $ \FP $ (see \cite{Lab2}). But let us remark that Labb\'e's
counter-example to the class P conjecture does not constitute a negative answer to the
original question (or remark) of Hof, Knill and Simon. In fact, it is readily verified that the
second shift
\[T^2(x)=abacacabacabacabacacabacabacacab\cdots\]
is the fixed point of the class P morphism:
\[g: \,\,\,\, a\mapsto ab,\,\,\,\, b\mapsto acac,\,\,\,\,  c\mapsto ac.\]
So the subshift generated by $x$ is in fact generated by a morphism in class P. The
morphism $g$ is closely related to the Toeplitz period-doubling word.

What is surprising is that Labb\'e's counter-example to the class P conjecture is not only
palindromic, but is as rich as possible in palindromes. More precisely, Droubay, Justin and
Pirillo observed that any finite word $u$ has at most $|u|+1$ distinct palindromic factors
(including the empty word). Accordingly, an infinite word $x$ is called rich if each factor
$u$ of $x$ has $|u|+1$ many distinct palindromic factors. It turns out that Labb\'e's
counter-example to the class P conjecture is a rich word. To see this,  we observe that $x$
is obtained from the fixed point $y$ of the morphism \[\tau: \,\,\,\,  b\mapsto ccb,\,\,\,\,
c\mapsto cb\] by inserting the symbol $a$ before every occurrence of each of the symbols
$b$ and $c.$ It is readily verified that  $\tau=\tau_3\tau_2\tau_1$ where  $\tau_1:
b\mapsto c, c\mapsto b$, $\tau_2: b\mapsto b, c\mapsto cb$ and $\tau_3: b\mapsto cb,
c\mapsto c,$ and hence by Corollary 6.3 and Proposition 6.6 in \cite{gjwz}  it follows that
$y$ is rich. Given that $y$ is rich, it now follows from Corollary 6.3  in \cite{gjwz} that $x$
is rich.

We regard the class P conjecture as an attempt to explain how a fixed point of a primitive
substitution can contain infinitely many palindromes. Of course, a typical substitution does
not preserve palindromes, hence one would expect that a palindrome generating
substitution would have some particular inherent structure. In this paper, we give such an
explanation in case $x$ is rich, or close to being rich. More precisely, the \emph{defect} of
a finite word $u,$  defined by $\text{D}(u)=|u|+1-|\textrm{Pal}(u)|,$ is a measure of the
extent to which $u$ fails to be rich. The \emph{defect} of a infinite word $x$ is defined by
$\text{D}(x)=\text{sup}\{\text{D}(u) | u\ \text{is a prefix of}\ x\}.$ This quantity can be finite
or infinite, and an infinite word is rich  if and only if its defect is equal to $0.$ Any infinite
word of finite defect is necessarily palindromic, but not conversely as is evidenced, for
example, by the Thue-Morse word. We show that if $y$ has finite defect and is generated by
a primitive substitution, then there exists a morphism $f\in \cP'$ and a rich word $x\in \FP$
such that $y=f(x).$  Actually, our result is more general as it applies as well to all primitive
morphic words (i.e., morphic images of fixed points of primitive substitutions). More
precisely:

\begin{thm}\label{main} Let $y$ be a primitive morphic word with finite defect.
Then there exists a morphism $f\in \cP'$ and a rich word $x\in \FP$ such that $y=f(x).$
\end{thm}

\noindent In this respect, every primitive morphic word $y$ with finite defect is generated
by not one, but two morphisms in $\cP'.$ The first which generates the fixed point $x$ in
Theorem~\ref{main} and the second which maps $x$ to $y.$ Labb\'e's counter-example
shows that even in the case of pure primitive morphic rich words, one cannot hope to have
$y$ itself in $\FP.$  A key ingredient in our proof is Durand's  characterization of primitive
morphic words in terms of the finiteness of the set of  derived words build from first
returns to prefixes (see \cite{Du}). A second involves a result of Balkov\'a et al.  in
\cite{BPS} linking words having finite defect with rich words.

\section{Preliminaries}

Given a finite non-empty set $A,$ we denote by $A^*$ the set of all finite words
$u=u_1u_2\cdots u_n$ with $u_i\in A.$ The quantity $n$ is called the length of $u$ and is
denoted  $|u|.$ The empty word, denoted $\varepsilon,$ is the unique element in $A^*$
with $|\varepsilon|=0.$ We set $A^+=A^*-\{\varepsilon\}.$ For each word $v\in A^+$, let
$|u|_v$  denote the number of occurrences of $v$ in $u$. We denote by $A^\omega$ the
set of all one-sided infinite words $x=x_0x_1x_2\cdots $ with $x_i\in A.$ Given $x\in
A^\omega,$ let $\Ff^+ (x)=\{x_ix_{i+1}\cdots x_{i+j}\,|\, i,j\geq 0\}$ denote the set of all
(non-empty) {\it factors} of $x.$ Recall that $x$ is called {\it recurrent} if each factor $u$
of $x$ occurs an infinite number of times in $x,$ and {\it uniformly recurrent} if for each
factor $u$ of $x$ there exists a positive integer $n$ such that $u$ occurs at least once in
every factor $v$ of $x$ with $|v|\geq n.$ An infinite word $x$ is called {\it periodic} if
$x=u^\omega$ for some $u\in A^+,$ and is called {\it ultimately periodic} if $x=vu^\omega$
for some $v\in A^*,$ and $u\in A^+.$ The word $x$ is called {\it aperiodic} if  $x$ is not
ultimately periodic. Let $x\in A^\omega$ and $u\in \Ff^+(x).$  A factor  $v$ of $x$ is called
a {\it first return} to $u$ in $x$ if $vu\in \Ff^+(x),$ $vu$ begins and ends in $u$ and
$|vu|_u=2.$  If $v$ is a first return to $u$ in $x,$ then $vu$ is called a {\it complete first
return} to $u$  in $x.$ We note that the two occurrences of $u$ in $vu$ may overlap. We
denote by $\R_u(x)$ the set of all first returns to $u$ in $x.$

A function $\tau: A\rightarrow A^+$ is called a {\it substitution}. A substitution $\tau$
extends by concatenation to a morphism from $A^*$ to $A^*$  and to a mapping from
$A^\omega$ to $A^\omega,$ i.e., $\tau(a_1a_2\cdots )=\tau(a_1)\tau(a_2)\cdots.$ By
abuse of notation we denote each of these extensions also by $\tau.$ A substitution  $\tau :
A\to A^+$ is \emph{primitive} if there exists a positive integer $N$ such that
$|\tau^N(a)|_b>0$ for all $ a,b\in A.$ A word $x\in A^\omega$ is a called a \emph{fixed
point} of a substitution $\tau$ if $\tau(x)=x.$ We say $x\in A^\omega$ is {\it pure primitive
morphic} if $x$ is a fixed point of some primitive substitution $\tau: A\rightarrow A^+.$   A
word $y\in B^\omega$ (where $B$ is a finite non-empty set) is called \emph{primitive
morphic} if there exists a  morphism $f:A^*\rightarrow B^*$ and a pure primitive morphic
word $x\in A^\omega$ with $y=f(x).$ It is readily verified that every primitive morphic
word is uniformly recurrent.

 In \cite{Du}, Durand obtains a nice characterization of primitive morphic words in
 terms of so-called derived words.  Let $x\in A^\omega$ be uniformly recurrent.
 Then  $\#\R_u(x)<+\infty$ for each $u\in \Ff^+(x).$
Let $u\in \Pre (x)$ be a non-empty prefix of $x.$ Then $x$ induces a linear order on
$\R_u(x)$ as follows: given distinct $v,v'\in \R_u(x)$ we declare $v<v'$ if the first
occurrence of $v$ in $x$ occurs prior to that of $v'.$ Let $A_u(x)=\{0,1,\ldots
,\#\R_u(x)-1\},$ and let $f_u:A_u(x)\rightarrow \R_u(x)$ denote the unique order preserving
bijection.  We can write $x$ uniquely as a concatenation of first returns to $u,$ i.e.,
$x=u_1u_2u_3\cdots$ with $u_i\in \R_u(x).$  Following \cite{Du} we define the {\it derived
word} of $x$ at $u,$ denoted $\D_u(x),$ as the infinite word with values in $ A_u(x)$ given
by
\[\D_u(x)=f_u^{-1}(u_1)f_u^{-1}(u_2)f_u^{-1}(u_3)\cdots .\]

\noindent For example, let \[x=01101001100101101001011001101001\cdots\] denote the
Thue-Morse word fixed by the substitution $0\mapsto 01, 1\mapsto 10.$ It is readily verified
that $\R_0(x)=\{011,01,0\}.$ So $A_0(x)=\{0,1,2\}$ and $f_0: A_0(x)\rightarrow \R_0(x)$ is
given by $f_0(0)=011, f_0(1)=01,$ and $f_0(2)=0.$ Writing $x$ as a concatenation of first
returns to $0$ we find
\[x=(011)(01)(0)(011)(0)(01)(011)(01)(0)(01)(011)(0)(011)(01)(0)\cdots\] and hence
\[\D_0(x)=012021012102012\cdots.\] It is readily verified that $\D_0(x)$ is the well known
Hall word.

\noindent The following result of Durand gives a  characterization of primitive morphic
words:

\begin{theorem}[F. Durand, Theorem 2.5 in \cite{Du}]\label{durand}
A word $x\in A^\omega$ is primitive morphic if and only if the set $\{\D_u(x)\,|\, u\in \Pre(x)\}$
is finite.
\end{theorem}

Given a finite or infinite word $u\in A^*$ we denote by $\textrm{Pal}(u)$ the set of all
palindromic factors of $u$ (including the empty word). Droubay, Justin and Pirillo proved
that any word $u\in A^*$ has at most $|u|+1$ many distinct palindromic factors including
the empty word (see \cite{djp}). A  finite word $u$ is called {\it rich} if  $\#\textrm{Pal}(u)=
|u|+1.$  For instance $aababbab$ is rich while $aababbaa$ is not. An infinite word is
called {\it rich} if all of its factors are rich. Rich words were first introduced by Glen et al.
in \cite{gjwz} and have been since studied in various papers.

The \emph{defect} of a finite word $u$ is defined by $\text{D}(u)=|u|+1-|\textrm{Pal}(u)|$.
The \emph{defect} of a infinite word $x$ is defined by $\text{D}(x)=\text{sup}\{\text{D}(u) |
u\ \text{is a prefix of}\ x\}.$ This quantity can be finite or infinite. Thus an infinite word is
rich  if and only if its defect is equal to $0.$

We will make use of the following result from \cite{gjwz} characterizing rich words
according to complete first returns.

\begin{theorem}[\cite{gjwz}, Theorem 2.14 and Remark 2.15.]\label{EJC}
An infinite word $x\in A^\omega$ is rich if and only if all complete first returns to any
palindromic factor in $x$ are themselves palindromes.
\end{theorem}

A morphism $f:A^*\rightarrow B^*$ is said to be of class P if there exists a palindrome
$p\in B^*$ and  palindromes $\{q_a\}_{a\in A}\subset  B^*$  such that $f(a)=pq_a$ for each
$a\in A.$ Let $\cP$ denote the set of all morphisms of class P. One problem with class P
morphisms is that they are not closed under composition. For instance, the
morphisms $f: 0\mapsto 0, 1\mapsto 01$ and $g: 0\mapsto 01, 1\mapsto 011$ are both in class P while the composition  $gf: 0\mapsto 01, 1\mapsto 01011$ is not.

Two morphisms $f,g:A^*\rightarrow B^*$ are said to be {\it conjugate} if there exists $u\in
B^*$ such that either $f(a)u=ug(a)$ for all $a\in A,$ or $uf(a)=g(a)u$ for all $a\in A.$ For
example, the morphisms $f: 0\mapsto 001, 1\mapsto 0010010010010$ is conjugate to the
morphism $g: 0\mapsto 010, 1\mapsto 0100100010010.$ In fact, taking $u=0010010,$ it is
readily verified that $f(a)u=ug(a)$ for $a\in \{0,1\}.$ Let $\cP'$ denote the set of all
morphisms $f:A^*\rightarrow B^*$  conjugate to some morphism in $\cP.$ We shall now
show that the class $\cP'$ is closed under composition. For this we need four lemmas the
first of which is a basic combinatorial result of words: 

\begin{lemma}[Proposition 1.3.4 in Lothaire~\cite{Lothaire}]\label{basic}
Let $w_1u=uw_2$ for words $w_1,w_2,u \in A^*$. Then there are words $x$ and $y$ such
that $w_1=xy$, $w_2=yx$ and $u=(xy)^kx$ for some non-negative $k \ge 0$.
\end{lemma}

\begin{lemma}\label{eqrel}
Conjugation of morphisms is an equivalence relation.
\end{lemma}

\begin{proof}
Clearly conjugation of morphisms is both reflexive and symmetric by definition. For
transitivity, suppose $f,g,h: A^* \to B^*$ are such that $f$ is conjugate to $g$, and
$g$ is conjugate to~$h$. We divide the proof to cases according to the two types in the
definition of conjugacy.

\smallskip
\noindent \textbf{Case 1}: Suppose there exist $u$ and $v$ in $B^*$ such that $f(a)u=ug(a)$
and $g(a)v=vh(a)$ for all $a \in A$. Then for all $a \in A$, we have $f(a)uv=ug(a)v=uvh(a)$
and hence $f$ is conjugate to $h$.

\smallskip
\noindent \textbf{Case 2}: Suppose there exist $u$ and $v$ such that $f(a)u=ug(a)$ and
$vg(a)=h(a)v$ for all $a \in A$.

We show first that $u$ is a suffix of $v$, or $v$ is a suffix of $u$. By Lemma~\ref{basic},
$f(a)u=ug(a)$  implies that $f(a)=xy$, $g(a)=yx$ and $u=(xy)^k x$ for some $x,y$ and $k \ge
0$. Similarly, $vg(a)=h(a)v$ implies that $g(a)=rs$, $h(a)=sr$ and $v= (sr)^t s$ for some
$r,s$ and $t\ge 0$. Now, $g(a)=yx=rs$, where by symmetry we can assume that $y=rw$ and
hence $s=wx$ for $w \in A^*$. Then $u= (xrw)^k x$ and $v=(wxr)^t wx=w(xrw)^tx$ and,
indeed one is a suffix of the other.

\smallskip

Without restriction we can assume that $u$ is a suffix of $v$, say $v=wu$. Then we have
\[
wf(a)u = wug(a)= vg(a)=h(a)v = h(a)wu,
\]
and hence $wf(a)=h(a)w$ for all $a \in A$. Hence $f$ is conjugate to $h$.

The other two cases (Case~3: $uf(a)=g(au)$ and $vg(a)=h(a)v$ for all $a \in A$; Case~4:
$uf(a)=g(a)u$ and $g(a)v=vh(a)$ for all $a \in A$) are obtained from Cases~1 and~2 by
interchanging~$f$ and~$h$ and using the symmetry condition of conjugation.
\end{proof}

The following lemma shows that the conjugation of morphisms is compatible with
composition.

\begin{lemma}\label{concat}
Let $f,f': A^*\to B^*$ and $g,g': B^*\to C^*$ be morphisms such that $f$ is
conjugate to $f'$ and $g$ is conjugate to $g'$. Then the compositions $g f$ and $g' f'$ are
conjugate.
\end{lemma}
\begin{proof}
Again we have cases to consider. Notice first that if $h,h': A^* \to B^*$ are any
functions that satisfy $h(a)x=xh'(a)$ for all $a \in A$ and some $x \in B^*$, then $h(w)x =
xh'(w)$ for all $w \in A^*$.

\smallskip
\noindent \textbf{Case 1}: Suppose there exist $u\in B^*$ and $v \in C^*$ such that
$f(a)u=uf'(a)$ and $g(a)v=vg'(a)$ for all $a \in A$. Then for all $a \in A$, we have
\[
gf(a)\cdot vg'(u) = vg'(f(a))g'(u) = vg'(f(a)u)=vg'(uf'(a))=vg'(u)\cdot g'f'(a),
\]
and hence $gf$ is conjugate to $g'f'$.

\smallskip
\noindent \textbf{Case 2}: Suppose there exist $u$ and $v$ such that $f(a)u=uf'(a)$ and
$vg(a)=g'(a)v$ for all $a \in A$. By Case~1, we have that $gf'$ is conjugate to $g'f$. Also,
$gf$ is conjugate to $gf'$, since $gf(a)\cdot g(u) = g(f(a)u) = g(uf'(a)) = g(u)\cdot gf'(a)$
and similarly $g'f$ is conjugate to $g'f'$. Hence, by transitivity, $gf$ is conjugate to $g'f'$.

Again, the other two cases follow from the above cases.
\end{proof}

\begin{lemma}\label{comp:lemma}
Let $f:  A^* \to B^*$ and  $g: B^* \to C^*$ be class $\cP$ morphisms. Then the
composition $gf$ is in $\cP'$.
\end{lemma}
\begin{proof}
Let $f(a)=pu_a$ for all $a \in A$ where $p$ and each $u_a$ is a palindrome, and $g(b)=
qv_b$ for all $b \in B$ where $q$ and each $v_a$ is a palindrome. If $p$ is empty, then
$u_a$ is nonempty, and $gf(a)=g(u_a)= q\cdot q^{-1}g(u_a)$ where $q$ and $q^{-1}g(u_a)$
are palindromes for all $a$. If $p$ is nonempty, then $gf(a) = g(pu_a)= g(p)g(u_a)$ and $gf$
is conjugate to $h$ defined by $h(a) =q^{-1}g(p)\cdot g(u_a)q$, if $u_a$ is non-empty, and
$h(a)=q^{-1}g(p)$ if $u_a$ is empty. The words $q^{-1}g(p)$ and $g(u_a)q$ are
palindromes, and hence $h \in \cP$, and $gf \in \cP'$ as required.
\end{proof}

\begin{proposition}\label{comp}
Class $\cP'$ is closed under composition.
\end{proposition}
\begin{proof}
Suppose $f:  A^* \to B^*$ and  $g: B^* \to C^*$ are in $\cP'$. Then there exist
class $\cP$ morphisms $f': A^*\to B^*$ and  $g': B^* \to C^*$ with $f$ conjugate
to $f'$ and $g$ conjugate to $g'$. By Lemma~\ref{concat}, the  composition $gf$ is
conjugate to $g'f$' and, by Lemma~\ref{comp:lemma}, the composition $g'f'$ is conjugate
to a class $\cP$ morphism $h$. By transitivity in Lemma~\ref{eqrel}, $gf$ is conjugate to
$h$, and hence by definition of $\cP'$', the composition $gf$ is in $\cP'$.
\end{proof}

In \cite{BPS}, Balkov\'a et al. introduce a related class of morphisms, denoted $\cPret,$
defined as follows: A morphism $f:A^*\rightarrow  B^*$ is in $\cPret$ if there exists a
palindrome $p$ such that for each $a\in A$ we have that $f(a)p$ is a palindrome, $f(a)p$
begins and ends in $p,$ $|f(a)p|_p=2,$ and $f(a)\neq f(b)$ whenever $a,b\in A$ with $a\neq
b.$ We call the palindrome $p$ the {\it marker}. For instance, the morphism $f: 0\mapsto
0, 1\mapsto 01$ is  in $\cP \cap \cPret.$ Here the marker is $p=0.$ In contrast, the
morphism $f:0\mapsto 00, 1\mapsto 01$ is in $\cP$ but not in $\cPret.$ While the morphism
$f: 0\mapsto 001, 1\mapsto 0010$ is in $\cPret$ (with marker $p=00100)$ but not in $\cP.$
They show that:

\begin{proposition}[Balkov\'a et al., Proposition 5.4 in \cite{BPS}]\label{subset}
$\cPret\subset \cP'.$ \end{proposition}

We note that while the definition of  "conjugacy" of two morphisms given in \cite{BPS} is
not the same as ours, their proof of Proposition 5.4 is consistent with our definition while
inconsistent with theirs. As it turns out, they intended for their definition to read the same
as ours \cite{BPS2}.

\section{Primitive morphic words of finite defect}

Let $\FP$ denote the set of all infinite words $x$ which are fixed by some primitive
substitution $f\in \cP'.$ We recall that the class P conjecture states that if $y$ is a
palindromic pure primitive morphic word, then $y\in \FP.$ Labb\'e's counter-example in
\cite{Lab2} shows that the conjecture as stated is false even if $y$ is rich. Instead, we
show:

\begin{theorem}\label{rich} Let $y\in A^\omega$ be a rich primitive morphic word.
Then there exists a morphism $g\in \cP'$ and a rich word $x\in \FP$ such that $y=g(x).$
\end{theorem}

\begin{proof} We can suppose without loss of generality that $A$ contains the symbol $0,$
and that $y$ begins in $0.$  Let $\R_0(y)$ denote the set of all first returns to $0$ in $y,$
$A_0(y)=\{0,1,\ldots ,\#\R_0(y)-1\},$ and
$f_0: A_0(y)\rightarrow R_0(y)$ be the unique order preserving bijection. Let
$f:A_0(y)^*\rightarrow A^*$ be the morphism defined by $f(a)=f_0(a)\in \R_0(y)\subset
A^+$ for each $a\in A_0(y).$ Writing $y=u_1u_2u_3\cdots $ with each $u_i\in \R_0(y),$ let
$\D_0(y)=f_0^{-1}(u_1)f_0^{-1}(u_2)f_0^{-1}(u_3)\cdots  $ denote the derived word of $y$
at the prefix $0.$ Thus $f(\D_0(y))=y.$

\noindent The next three lemmas are stated in terms of the prefix $0$ of $y$ since they are
needed only in this special case. But in fact they  hold for all palindromic prefixes $u$ of
$y.$

\begin{lemma}\label{morph} The morphism $f:A_0(y)^*\rightarrow A^*$  is in $\cP$
and thus in $\cP'.$
\end{lemma}
\begin{proof} Since $y$ is rich, by Theorem~\ref{EJC} it follows that for each $v\in \R_0(y)$
there is a palindrome $v'\in A^*$ (possibly empty) such that $v=0v'.$ Thus, for each $a\in A, $
there exists a palindrome $v_a$ such that
$f(a)=f_0(a)=0v_a.$ Hence $f\in \cP\subset \cP'.$
\end{proof}

\begin{lemma} $\D_0(y)\in A_0(y)^\omega$ is rich and begins in $0.$
\end{lemma}

\begin{proof} Since $f_0$ is order preserving, $f_0^{-1}(u_1)=0,$ whence $\D_0(y)$
begins in $0.$ Let $z$ be a complete first return in $\D_0(y)$ to a palindrome
$u\in \Ff^+(\D_0(y)).$ By Lemma~\ref{morph} we deduce that $f(u)0$ is a palindromic
factor of $y$ and $f(z)0$ is a complete first return in $y$ to $f(u)0.$ Since $y$ is
rich it follows from Theorem~\ref{EJC} that $f(z)0$ is a palindrome. By Lemma~\ref{morph}
and item 3. in Remark 5.2 in \cite{BPS} we deduce that $z$ is a palindrome, and hence
$\D_0(y)$ is rich by Theorem~\ref{EJC}.
\end{proof}

\begin{lemma} $\D_0(y)$ is primitive morphic.
\end{lemma}

\begin{proof} In item 5.\ of Proposition 2.6 of \cite{Du}, it is shown that every derived
word of $\D_0(y)$ is also a derived word of $y.$ Since $y$ is primitive morphic,
it follows from Theorem~\ref{durand} that $y$ has  only finitely many distinct derived words,
and hence $\D_0(y)$ has finitely many distinct derived words, and hence by Theorem~\ref{durand}
$\D_0(y)$ is primitive morphic.

\end{proof}

Combining the three previous lemmas we deduce that if $y\in A^\omega$ is a rich primitive
morphic word beginning in $0,$ then $\D_0(y)\in A_0(y)^\omega$ is a rich primitive
morphic word beginning in $0,$ and if $y=u_1u_2u_3\cdots $ with $u_i\in \R_0(y),$ then
$\D_0(y)=f^{-1}(u_1)f^{-1}(u_2)f^{-1}(u_3)\cdots  $ where $f: A_0(y)^*\rightarrow A^*$
belongs to $\cP'.$

Thus, we can inductively define a sequence of infinite words $(S_n(y))_{n\geq 0}$ with
values in finite sets $(\A_n)_{n\geq 0}$  by $S_0(y)=y,$ and $\A_0=A,$ and for $n\geq 0:$
$S_{n+1}(y)=\D_0(S_n(y))$ and $\A_n=A_0(S_n(y)).$ Moreover, for each $n\geq 1$ there
exists a morphism $g_n: \A_{n}^* \rightarrow \A_{n-1}^*$ in $\cP'$ such that writing
$S_{n-1}(y)=u_1u_2u_3\cdots $ with $u_i\in \R_0(S_{n-1}(y)),$ we have
$S_{n}(y)=g_n^{-1}(u_1)g_n^{-1}(u_2)g_n^{-1}(u_3)\cdots  .$ In other words, the sequence
$(S_n(y))_{n\geq 0}$ is just the sequence of iterated derived words of $y$ corresponding
each time to the prefix $0.$

By Theorem~\ref{durand} there exist $0\leq m<n$ such that $S_m(y)=S_n(y).$ Let
$x=S_m(y).$ Let $h=g_{m+1}g_m \cdots g_n.$ Then $h(x)=x$ and by the proof of
Proposition 3.3 in \cite{Du} we deduce  that $h$ is a primitive substitution. By
Proposition~\ref{comp} the morphism $h$ is in  $\cP'.$  Thus $x\in \FP.$ Finally let
$g=g_mg_{m-1} \cdots g_1.$ Then $y=g(x)$ and by Proposition~\ref{comp}  we deduce
that $g\in \cP'$ as required. This completes the proof of Theorem~\ref{rich}.
\end{proof}

\begin{corollary}Let $z\in A^\omega$ be a primitive morphic word having finite defect.
Then there exists a morphism $g\in \cP'$ and a rich word $x\in \FP$ such that $z=g(x).$
\end{corollary}

\begin{proof} In Theorem 5.5 in \cite{BPS},  the authors show that if $z\in A^\omega$ is
a  uniformly recurrent word of finite defect, then there exists a rich word $y\in B^\omega$
and a morphism $f:B^*\rightarrow A^*$ in $\cPret$ such that $z=f(y).$ In the proof of the
theorem, it is revealed that $y$ is actually a derived word of $z.$
Thus, if $z$ is primitive morphic, then by Theorem~\ref{durand} so is $y,$ and hence by
Theorem~\ref{rich}  there exists a rich word $x\in \FP$ and a morphism $h\in \cP'$ such
that $y=h(x).$ Let $g=fh.$ Then $z=g(x)$ and by  Proposition~\ref{comp} and
Proposition~\ref{subset} we deduce that $g\in \cP'.$
\end{proof}

We end with an illustration applied to Labb\'e's example. Let
\[y=acabacacabacabacabacacabac\cdots\] be the fixed point of the morphism
 $a\mapsto ac, b\mapsto acab, c\mapsto ab.$ Then $\R_a(y)=\{ac,ab\}$ and the derived word
 $\D_a(y)\in \{0,1\}^\omega$ is the fixed point of the morphism $0\mapsto 01, 1\mapsto 001$
 which is clearly in $\cP'.$ Thus, setting $x=\D_a(y),$ we have that $x\in \FP$ and $y=f(x)$
 where $f: 0\mapsto ac, 1\mapsto ab$ is in $\cP'.$

\end{document}